\documentclass[oneside,english]{amsart}
\usepackage{libertine}
\usepackage[T1]{fontenc}
\usepackage[latin9]{inputenc}
\usepackage{babel}
\usepackage{pifont}
\usepackage{mathtools}
\usepackage{enumitem}
\usepackage{amstext}
\usepackage{amsthm}
\usepackage{amssymb}
\usepackage{setspace}
\usepackage[unicode=true,pdfusetitle,
 bookmarks=true,bookmarksnumbered=false,bookmarksopen=false,
 breaklinks=false,pdfborder={0 0 0},pdfborderstyle={},backref=false,colorlinks=false]
 {hyperref}

\makeatletter
\numberwithin{equation}{section}
\numberwithin{figure}{section}
\theoremstyle{plain}
\newtheorem{thm}{\protect\theoremname}
\theoremstyle{plain}
\newtheorem{lem}[thm]{\protect\lemmaname}
\theoremstyle{remark}
\newtheorem{rem}[thm]{\protect\remarkname}
\theoremstyle{definition}
\newtheorem{defn}[thm]{\protect\definitionname}
\theoremstyle{remark}
\newtheorem*{rem*}{\protect\remarkname}
\theoremstyle{remark}
\newtheorem{claim}[thm]{\protect\claimname}
\newlist{casenv}{enumerate}{4}
\setlist[casenv]{leftmargin=*,align=left,widest={iiii}}
\setlist[casenv,1]{label={{\itshape\ \casename} \arabic*.},ref=\arabic*}
\setlist[casenv,2]{label={{\itshape\ \casename} \roman*.},ref=\roman*}
\setlist[casenv,3]{label={{\itshape\ \casename\ \alph*.}},ref=\alph*}
\setlist[casenv,4]{label={{\itshape\ \casename} \arabic*.},ref=\arabic*}
\theoremstyle{plain}
\newtheorem{cor}[thm]{\protect\corollaryname}

\AtBeginDocument{

}

\makeatother

\providecommand{\casename}{Case}
\providecommand{\claimname}{Claim}
\providecommand{\corollaryname}{Corollary}
\providecommand{\definitionname}{Definition}
\providecommand{\lemmaname}{Lemma}
\providecommand{\remarkname}{Remark}
\providecommand{\theoremname}{Theorem}

\makeatletter
\@namedef{subjclassname@2020}{%
	\textup{2020} Mathematics Subject Classification}
\makeatother

\begin{document}
\global\long\def\fii{\varphi}%
\global\long\def\T{\mathbb{T}}%
\global\long\def\F{\mathbb{F}}%
\global\long\def\M{\mathcal{M}}%
\global\long\def\zfc{\mathrm{ZFC}}%
\global\long\def\N{\mathcal{N}}%
\global\long\def\p{\mathbb{P}}%
\global\long\def\q{\mathbb{Q}}%
\global\long\def\NN{\mathbb{N}}%
\global\long\def\RR{\mathbb{R}}%
\global\long\def\power{\mathcal{P}}%
\global\long\def\LL{\mathcal{L}}%
\global\long\def\K{\mathcal{K}}%
\global\long\def\A{\mathcal{A}}%

\global\long\def\hod{\mathrm{HOD}}%
\global\long\def\Def{\operatorname{Def}}%
\global\long\def\zfc{\mathrm{ZFC}}%
\global\long\def\im{\operatorname{Im}}%
\global\long\def\dom{\operatorname{Dom}}%
\global\long\def\rng{\operatorname{Range}}%
\global\long\def\ord{{\bf Ord}}%
\global\long\def\lp{\operatorname{lp}}%
\global\long\def\meas{\operatorname{meas}}%
\global\long\def\card{{\bf Card}}%
\global\long\def\lex{\mathrm{Lex}}%
\global\long\def\id{\operatorname{Id}}%
\global\long\def\ult{\operatorname{Ult}}%
\global\long\def\It{\operatorname{It}}%

\global\long\def\th#1{\mathrm{Th}\left(#1\right)}%
\global\long\def\crit#1{\mathrm{crit}\left(#1\right)}%
\global\long\def\cof#1{\mathrm{cf}\left(#1\right)}%
\global\long\def\fin{\power_{\mathrm{Fin}}}%
\global\long\def\cfq{Q^{\mathrm{cf}}}%
\global\long\def\otp{\operatorname{otp}}%
\global\long\def\fin#1{\left[#1\right]^{<\omega}}%
\global\long\def\psu#1#2{\prescript{#1}{}{#2}}%
\global\long\def\supp{\operatorname{supp}}%

\global\long\def\con{\subseteq}%
\global\long\def\til{,\dots,}%
\global\long\def\emp{\varnothing}%
\global\long\def\smin{\mathord{\smallsetminus}}%
\global\long\def\aa{\mathtt{aa}}%
\global\long\def\mets{\mathord{\upharpoonright}}%

\global\long\def\sdiff{\triangle}%
\global\long\def\po{\trianglelefteq}%
\global\long\def\spo{\vartriangleleft}%
\global\long\def\pdwn{\mathord{\downarrow}}%
\global\long\def\pup{\mathord{\uparrow}}%
\global\long\def\nec{\square}%
\global\long\def\pos{\lozenge}%
\global\long\def\necc{\square_{ccc}}%
\global\long\def\posc{\lozenge_{ccc}}%

\title[Absoluteness for the theory of $C_{<\lambda_{1}\til<\lambda_{n}}^{*}$]{Absoluteness for the theory of the inner model\linebreak{}
 constructed from finitely many cofinality quantifiers}
\author{Ur Ya'ar}
\address{Einstein Institute of Mathematics \newline Hebrew University of Jerusalem\newline
Edmond J. Safra Campus, Givat Ram \newline  Jerusalem 91904, ISRAEL}
\email{ur.yaar@mail.huji.ac.il}
\thanks{I would like to thank my advisor, Prof. Menachem Magidor, for his
guidance and support without which this work would not have been possible.}
\begin{abstract}
We prove that the theory of the models constructible using finitely
many cofinality quantifiers -- $C_{\lambda_{1}\til\lambda_{n}}^{*}$
and $C_{<\lambda_{1}\til<\lambda_{n}}^{*}$ for $\lambda_{1}\til\lambda_{n}$
regular cardinals -- is set-forcing absolute under the assumption
of class many Woodin cardinals, and is independent of the regular
cardinals used. Towards this goal we prove some properties of the
generic embedding induced from the stationary tower restricted to
$<\mu$-closed sets.
\end{abstract}

\maketitle

\section{Introduction and preliminaries}

Following the general framework set by Kennedy, Magidor and V\"{a}\"{a}n\"{a}nen
 in \cite{IMEL} for inner models constructed from extended logics,
our aim is to investigate $C_{\lambda_{1}\til\lambda_{n}}^{*}$ and
$C_{<\lambda_{1}\til<\lambda_{n}}^{*}$ -- the models of sets constructible
using the logics $\LL(Q_{\lambda_{1}}^{\mathrm{cf}}\til Q_{\lambda_{n}}^{\mathrm{cf}})$
and $\LL(Q_{<\lambda_{1}}^{\mathrm{cf}}\til Q_{<\lambda_{n}}^{\mathrm{cf}})$
respectively, where $Q_{\lambda}^{\mathrm{cf}}$ (resp. $Q_{<\lambda}^{\mathrm{cf}}$)
is the quantifier asserting that an ordinal has cofinality $\lambda$
(resp. $<\lambda$) and $\lambda_{1}<\dots<\lambda_{n}$ are regular
cardinals. Kennedy, Magidor and V\"{a}\"{a}n\"{a}nen have proved that,
assuming the existence of a proper class of Woodin cardinals, the
theory of $C^{*}=C_{\omega}^{*}=C_{<\omega_{1}}^{*}$ is set-forcing
absolute, and equals the theory of $C_{<\kappa}^{*}$ for any regular
$\kappa$ \cite[theorem 5.18]{IMEL}. Our goal it to generalize this
theorem to $C_{<\lambda_{1}\til<\lambda_{n}}^{*}$ for any regular
uncountable $\lambda_{1}\til\lambda_{n}$, from which also the case
of $C_{\lambda_{1}\til\lambda_{n}}^{*}$ can be deduced. To obtain
this result, we use a variation of Woodin's stationary tower -- the
$<\mu$-closed stationary tower, introduced by Foreman and Magidor
in \cite{foreman-magidor}. This tower has the property that it does
not change the notion of ``being of cofinality $<\mu$''. We begin
by stating and proving some facts regarding this tower, and then prove
the main theorem. We end by showing that this method cannot be simply
pushed to the case of infinitely many cofinalities, which leaves open
the question of absoluteness for theory of this kind of models.

Our notation will mostly follow \cite{foreman-magidor}. A set $S\con\power(Y)$
is called \emph{stationary in $Y$} if for every algebra $\mathfrak{A}=\left\langle Y,f_{i}\right\rangle _{i\in\omega}$
there is $Z\in S$ such that $Z$ is a subalgebra of $\mathfrak{A}$.
The collection of non-stationary subsets of $Y$ is an ideal denoted
by $\mathrm{NS}(Y)$. If $S$ is stationary in $Y$ then we denote
the restriction of the non-stationary ideal to $S$ by $\mathrm{NS}(Y)\mets S$.
We say that a set $S$ is simply \emph{stationary} if it is stationary
in $\bigcup S$. If $\emp\ne X\con Y$, then for $S\con\power\left(Y\right)$
we define its projection to $X$ by $S\pdwn X=\left\{ Z\cap X\mid Z\in S\right\} $
and for $T\con\power\left(X\right)$ we define its lift to $Y$ by
$T\pup Y=\left\{ Z\in\power\left(Y\right)\mid Z\cap X\in T\right\} $
.

A sequence $\mathcal{I}=\left\langle \mathcal{I}_{\beta}\mid\beta<\kappa\right\rangle $
such that each $\mathcal{I}_{\beta}$ is an ideal on $\power(H(\beta))$
is called a \emph{tower of ideals} if for every $\alpha<\beta$, if
for any $S\con\power(H(\beta))$, $S\notin\mathcal{I}_{\beta}\implies S\pdwn H(\alpha)\notin\mathcal{I}_{\alpha}$
and for any $T\con\power(H(\alpha))$, $T\notin\mathcal{I}_{\alpha}\implies T\pup H(\beta)\notin\mathcal{I}_{\beta}$
(i.e. positive sets project/lift to positive sets). Any tower of ideals
$\mathcal{I}$ gives rise to a Boolean algebra $b(\mathcal{I})$ such
that a generic $G\con b(\mathcal{I})$ induces a generic embedding
$j_{G}:V\to Ult(V,G)\con V\left[G\right]$ (for more details see \cite{larson-stationary,foreman-magidor}).
The tower is called \emph{precipitous} if every such generic embedding
is well-founded. For a strongly inaccessible cardinal $\delta$, \emph{the
(full) stationary tower on $\delta$} is $\left\langle \mathrm{NS}(H(\beta))\mid\beta<\delta\right\rangle $,
and a \emph{restricted tower} is of the form $\left\langle \mathrm{NS}(H(\beta))\mets S_{\beta}\mid\beta<\delta\right\rangle $
where each $S_{\beta}$ is stationary in $H(\beta)$. The Boolean
algebra $b(\left\langle \mathrm{NS}(H(\beta))\mid\beta<\delta\right\rangle )$
of the full stationary tower on $\delta$ is forcing equivalent to
as the poset 
\[
\p_{<\delta}=\left\{ a\in V_{\delta}\mid a\,\text{is stationary in }\text{\ensuremath{\cup a}}\right\} 
\]
with $a\geq b$ iff $\cup a\con\cup b$ and $\forall Z\in b$ $Z\cup\left(\cup a\right)\in a$
(i.e. $b\pdwn\left(\cup a\right)\con a$). Restricted stationary towers
on $\delta$ usually correspond to subsets of $\p_{<\delta}$ with
the same order. We will not distinguish the two notions. 

Given a (perhaps restricted) tower on $\delta$ $\p$ and a generic
$G\in\p$, members of $Ult(V,G)$ can be represented as equivalence
classes $\left[f\right]_{G}$ for some $f:a\to V$ where $a\in G$.
For every $x\in V$ $j_{G}(x)$ is represented by the constant function
$c_{x}$ (on any $a\in G$), and every $\alpha<\delta$ is represented
by the function $Z\mapsto\otp(Z\cap\alpha)$ (cf. \cite{larson-stationary}).
\begin{lem}
\label{lem:inacc}Let $\delta$ be a Woodin cardinal and $\rho>\delta$
strongly inaccessible. Let $\p$ be a (perhaps restricted) precipitous
stationary tower on $\delta$, $G\con\p$ generic and $j:V\to M$
the derived embedding. Then:
\end{lem}

\begin{enumerate}
\item \label{enu:fix-inacc} $j\left(\rho\right)=\rho$.
\item \label{lem:Cof<inacc}For every ordinal $\eta$ the following are
equivalent:

\emph{(a) $V\vDash\cof{\eta}<\rho$ ~~~(b) $V\left[G\right]\vDash\cof{\eta}<\rho$
~~~(c) $M\vDash\cof{\eta}<\rho$.}

\end{enumerate}
\begin{proof}
For (1): If $\eta<\rho$ then every $x\in j\left(\eta\right)$ is
of the form $\left[f\right]_{G}$ for $f:X\to\eta$ where $X\in V_{\delta}$,
so there are at most $\psu{\delta}{\eta}<\rho$ many such functions,
hence $j(\eta)<\rho$. Similarly every $x\in j\left(\rho\right)$
is of the form $\left[f\right]_{G}$ for $f:X\to\rho$ where $X\in V_{\delta}$,
but now $\left|X\right|<\rho$ so $f$ is bounded by some $\eta<\rho$,
so $\left[f\right]_{G}<j\left(\eta\right)$. Hence $j(\rho)=\sup\left\{ j(\eta)\mid\eta<\rho\right\} \leq\rho$. 

For (2): (c)$\Rightarrow$(b) since a cofinal sequence in $M$ remains
so in $V\left[G\right]$. 

(b)$\Rightarrow$(a) since $\rho>\delta=\left|\p\right|$ so by $\delta^{+}$-c.c.
cofinalities $\geq\rho$ are preserved from $V$ to $V\left[G\right]$. 

For (a)$\Rightarrow$(c) assume towards contradiction that $\mathrm{cf}^{V}\left(\eta\right)=\lambda<\rho$
while $\mathrm{cf}^{M}\left(\eta\right)\geq\rho$. Let $a\in\p$ such
that $a\Vdash\mathrm{cf}^{M}\left(\eta\right)\geq\rho$ and $f_{\eta}:a\to V$
represent $\eta$ (without loss of generalisation $f_{\eta}\left(X\right)\leq\eta$
for every $X\in a$, since $\eta\leq j\left(\eta\right)$). Let $\left\langle x_{\alpha}\mid\alpha<\lambda\right\rangle \in V$
be cofinal in $\eta$. For $X\in a$, let $g\left(X\right)$ be the
supremum of all values $f\left(X\right)$ for some function $f\in V$
with domain $b\leq a$ which may represent some $x_{\alpha}$. More
precisely -- for each $b\leq a$ choose some $f_{b,\alpha}:b\to V$
such that $b\Vdash\left[f_{b,\alpha}\right]_{G}=\check{x}_{\alpha}$,
and let 
\[
g\left(X\right)=\sup\left\{ f_{b,\alpha}\left(Y\right)\mid b\leq a,\alpha<\lambda,Y\in b\text{ s.t. }Y\cap\left(\cup a\right)=X\right\} 
\]
For every $\alpha<\lambda$ there are at most $\delta$ such functions,
each with domain of size $<\delta$, so all-in-all there are at most
$\lambda\cdot\delta$ possible values. Hence $\cof{g\left(X\right)}\leq\lambda\cdot\delta<\rho$
(by strong inaccessibility) for all $X\in a$, so $a\Vdash\cof{\left[g\right]}<j(\rho)=\rho$\,.
Since for every $\alpha$ $a\Vdash\check{x}_{\alpha}<\left[f_{\eta}\right]_{G}$,
by the definition of $g$ we get that $\left[g\right]\leq\left[f_{\eta}\right]=\eta$,
and since $\mathrm{cf}^{M}\left(\eta\right)\geq\rho$\,, $\left[g\right]<\eta$\,.
But by the construction for every $b\leq a$ and $\alpha<\lambda$,
$b\Vdash\check{x}_{\alpha}=\left[f_{b,\alpha}\right]_{G}\leq\left[g\right]_{G}$,
so $a\Vdash\check{x}_{\alpha}\leq\left[g\right]_{G}$ for every $\alpha$,
so $\left\langle x_{\alpha}\mid\alpha<\lambda\right\rangle $ is bounded
bellow $\eta$, contradicting the assumption that it is cofinal in
$\eta$.
\end{proof}

\section{The $<\mu$-closed stationary tower}

In \cite{foreman-magidor} we have the following theorem (theorem
1.3):
\begin{thm}
Let $\delta$ Woodin, $\mu<\lambda<\delta$ regular, for every strong
limit $\gamma$, let \emph{
\begin{align*}
S_{\gamma} & =\left\{ Z\in\power_{\lambda}\left(H\left(\gamma\right)\right)\mid Z\cap\lambda\in\lambda\land Z\cap\gamma\text{ is }<\mu\text{ closed}\right\} \\
\mathrm{NS}\left(\lambda,\gamma\right) & =\text{non-stationary subsets of }\power_{\lambda}\left(H\left(\gamma\right)\right)\\
\mathcal{I} & =\left\{ \mathrm{NS}\left(\lambda,\gamma\right)\restriction S_{\gamma}\mid\gamma<\delta\right\} 
\end{align*}
}then $\mathcal{I}$ is a tower of ideals and $b\left(\mathcal{I}\right)$
is $\delta$-presaturated. 
\end{thm}

\begin{rem}
$b\left(\mathcal{I}\right)$ is forcing equivalent to $\p\left(\delta,\lambda,<\mu\text{-closed}\right)$,
namely 
\[
\left\{ a\in V_{\delta}\mid a\con\power_{\lambda}\left(\cup a\right)\text{ is stationary s.t. \ensuremath{\forall Z\in a}, \ensuremath{Z\cap\ord} is \ensuremath{<\mu}-closed}\land Z\cap\lambda\in\lambda\right\} 
\]
with the relation defined above, $b\leq a$ iff $\cup a\con\cup b$
and $b\pdwn\left(\cup a\right)\con a$.
\end{rem}

\begin{thm}
\label{thm:embedding-properties}If $G\con b\left(\mathcal{I}\right)$
is generic then the induced generic embedding $j:V\to M$ satisfies:
\end{thm}

\begin{enumerate}
\item $M$ is well-founded and $\psu{<\delta}M\cap V\left[G\right]\con M$.
\item \label{enu:crit}$\crit j=\lambda$ and $j\left(\lambda\right)=\delta$.
\item \label{enu:cof-preserve}For any $\theta$, 
\begin{enumerate}
\item If $V\vDash\cof{\theta}<\lambda$ then $\mathrm{cf}^{V}\left(\theta\right)=\mathrm{cf}^{V\left[G\right]}\left(\theta\right)=\mathrm{cf}^{M}\left(\theta\right)$;
\item If $V\vDash\cof{\theta}\geq\mu$ then $V\left[G\right],M\vDash\cof{\theta}\geq\mu$
;
\item $V\vDash\cof{\theta}<\delta$ iff $V\left[G\right]\vDash\cof{\theta}<\delta$
iff $M\vDash\cof{\theta}<\delta$ .
\end{enumerate}
\end{enumerate}
To prove this theorem we introduce some notions defined in \cite{foreman-magidor}
and prove variations of two lemmas.
\begin{defn}
Fix regular cardinals $\mu$, $\lambda$ and $\theta\gg\lambda$,
and let $\mathfrak{A}=\left\langle H(\theta),\in,\triangle,\mu,f_{i}\right\rangle _{i\in\omega}$
be a Skolemized algebra on $H(\theta)$ where the $f_{i}$s are closed
under composition and $\triangle$ is a well-order. Define by recursion
a sequence $\left\langle \left(\mathfrak{A}_{i},F_{i},G_{i},F_{i}^{*},G_{i}^{*}\right)\mid i\leq\mu\right\rangle $
of functions and expansions of $\mathfrak{A}$ as follows. $\mathfrak{A}_{0}=\mathfrak{A}$,
$F_{0}^{*}=G_{0}^{*}=\emp$. For every $i$, we will define $\mathfrak{A}_{i}=\left\langle \mathfrak{A},F_{i}^{*},G_{i}^{*}\right\rangle $
and
\begin{eqnarray*}
F_{i}:\theta\times H(\theta)\to\theta &  & F_{i}(\xi,x)=\begin{cases}
\sup\left(\mathrm{Sk}^{\mathfrak{A}_{i}}(x\cap\theta)\right) & \xi=0\\
\sup\left(\mathrm{Sk}^{\mathfrak{A}_{i}}(x\cap\xi)\right) & \xi\ne0
\end{cases}\\
G_{i}:i\times H(\theta)\to H(\theta) &  & G_{i}(j,Z)=(G_{j}\mets Z,F_{i}\mets Z).
\end{eqnarray*}
At stage $i>0$ we define
\begin{eqnarray*}
F_{i}^{*}:i\times\theta\times H(\theta)\to\theta &  & F_{i}^{*}(j,\xi,x)=F_{j}(\xi,x)\\
G_{i}^{*}:i\times i\times H(\theta)\to H(\theta) &  & G_{i}^{*}(j,k,x)=\begin{cases}
G_{j}(k,x) & k<j\\
\emp & k\geq j
\end{cases}
\end{eqnarray*}
 Denote $\bar{F}=F_{\mu}^{*}$, $\bar{G}=G_{\mu}^{*}$, $\bar{\mathfrak{A}}=\mathfrak{A}_{\mu}=\left\langle \mathfrak{A},\bar{F},\bar{G}\right\rangle $.
\end{defn}

\begin{lem}[{Variation of \cite[Lemma 1.6]{foreman-magidor}}]
\label{lem:mu-closed-extension}Let $\gamma_{0}<\gamma_{1}$ be strong
limit cardinals of cofinality $\geq\mu$. Let $\mathfrak{A}$ be a
Skolemized algebra on $H(\gamma_{1})$ expanding $\left\langle H(\gamma_{1}),\in,\triangle,\left\{ \gamma_{0}\right\} ,\left\{ \mu\right\} \right\rangle $
where $\triangle$ is a well order on $H(\gamma_{1})$ . Then for
every $x\con H(\gamma_{0})$ which is $<\mu$-closed, if $\bar{x}\con H(\gamma_{1})$
is the subalgebra of $\bar{\mathfrak{A}}$ generated by $x$ and $\bar{x}\cap H(\gamma_{0})=x$
then $\bar{x}\cap\gamma_{1}$ is $<\mu$-closed in $\gamma_{1}$.
\begin{rem*}
The lemma in \cite{foreman-magidor} is stated for $\leq\mu$-closed
sets rather than for $<\mu$-closed sets.
\end{rem*}
\end{lem}

\begin{proof}
Assume $\bar{x}$ is not $<\mu$-closed and let $\xi\in\bar{x}$ be
least such that $\bar{x}\cap\xi$ is not $<\mu$ closed (or $\gamma_{1}$
if there is no such $\xi$). By assumption $\xi>\gamma_{0}$. So there
is some sequence $\left\langle a_{\beta}\mid\beta<\nu\right\rangle \con\bar{x}\cap\xi$
for some $\nu<\mu$ such that $\sup\left\langle a_{\beta}\mid\beta<\nu\right\rangle \notin\bar{x}\cap\xi$.
Note that by minimality, $\left\langle a_{\beta}\mid\beta<\nu\right\rangle $
is cofinal in $\bar{x}\cap\xi$. We first claim that $\xi$ is regular
(or $\gamma_{1}$). Otherwise, if $\xi\in\bar{x}$ and is singular,
since $\bar{x}\prec\left\langle H(\gamma_{1}),\in,\triangle\right\rangle $\emph{
}$\bar{x}$ knows that $\xi$ is singular, so there is some $\xi'\in\bar{x}$,
$\xi'<\xi$, and a sequence $\left\langle \xi_{\alpha}\mid\alpha<\xi'\right\rangle \in\bar{x}$
cofinal in $\bar{x}\cap\xi$. So both $\left\langle \xi_{\alpha}\mid\alpha<\xi'\right\rangle $
$\left\langle a_{\beta}\mid\beta<\nu\right\rangle $ are cofinal in
$\bar{x}\cap\xi$ so by letting $\alpha_{\beta}$ be the first $\alpha<\xi'$
such that $\xi_{\alpha}>a_{\beta}$ we get a sequence $\left\langle \alpha_{\beta}\mid\beta<\nu\right\rangle $
cofinal in $\bar{x}\cap\xi'$ by contradiction to the minimality of
$\xi$.

Second, let $\left\langle a_{\beta}\mid\beta<\nu\right\rangle \con\bar{x}\cap\xi$
for some $\nu<\mu$ such that $\sup\left\langle a_{\beta}\mid\beta<\nu\right\rangle \notin\bar{x}\cap\xi$.
For every $\beta<\nu$ there are an $\bar{\mathfrak{A}}$-term $\tau_{\beta}$
and parameters $\vec{\alpha}_{\beta}\in\left(x\cap\gamma_{0}\right)^{<\omega}$
such that $a_{\beta}=\tau_{\beta}(\vec{\alpha}_{\beta})$. By the
construction of $\bar{\mathfrak{A}}$, for every $\beta<\nu$ there
is some $\rho_{\beta}\in\text{\ensuremath{\mu\cap\bar{x}}}$ such
that $a_{\beta}=\tau_{\beta}\left(\vec{\alpha}_{\beta}\right)<\bar{F}(\rho_{\beta},\xi,\gamma_{0})$
(or $<\bar{F}(\rho_{\beta},0,\gamma_{0})$ if $\xi=\gamma_{1}$).
By assumption $\mu\cap\bar{x}=\mu\cap x$, and since $\nu<\mu$ and
$x$ is $<\mu$ closed, there is $\rho\in x$ such that for every
$\beta<\nu$, $a_{\beta}<\bar{F}(\rho,\xi,\gamma_{0})$ (or $<\bar{F}(\rho,0,\gamma_{0})$\,).
But since $\xi>\gamma_{0}>\mu$ and is regular, $\bar{F}(\rho,\xi,\gamma_{0})<\xi$
(or $\bar{F}(\rho,0,\gamma_{0})<\gamma_{1}$) in contradiction to
the minimality of $\xi$.
\end{proof}
\begin{lem}[{Variation of \cite[Lemma 1.7]{foreman-magidor}}]
\label{lem:union-mu-closed}Let $\theta$ be a regular cardinal $>\mu$.
Let $\mathfrak{A}$ be a Skolemized algebra on $H(\theta)$ expanding
$\left\langle H(\theta),\in,\triangle\right\rangle $. Let $\left\langle x_{\alpha}\mid\alpha<\eta\right\rangle $
be a continuous increasing sequence of $<\mu$-closed substructures
of $\bar{\mathfrak{A}}$ of cardinality $\geq\mu$ and let $\left\langle \gamma_{\alpha}\mid\alpha<\eta\right\rangle \con x_{0}\smin\mu$
be an increasing sequence of cardinals closed in $\gamma=\sup\left\langle \gamma_{\alpha}\mid\alpha<\eta\right\rangle $.
Suppose that for every $\alpha<\eta$, 
\begin{enumerate}
\item $x_{\alpha+1}$ is a $\gamma_{\alpha+1}$-end-extension of $x_{\alpha}$
i.e $x_{\alpha+1}\cap\sup(x_{\alpha}\cap\gamma_{\alpha+1})=x_{\alpha}\cap\gamma_{\alpha+1}$
,
\item $x_{\alpha}$is the substructure of $\bar{\mathfrak{A}}$ generated
by $x_{\alpha}\cap\gamma_{\alpha}$.
\end{enumerate}
Then $z=\bigcup_{\alpha<\eta}x_{\alpha}$ is $<\mu$-closed.

\end{lem}

\begin{proof}
Assume towards contradiction that $z$ is not $<\mu$ closed  and
let $\left\langle \beta_{\alpha}\mid\alpha<\nu\right\rangle $ witness
this. Note that there is no $\gamma$ such that $x_{\gamma}$ contains
unboundedly many $\beta_{\alpha}$s, since $x_{\gamma}$ is $<\mu$
closed, so we can find a strictly increasing subsequence $\left\langle \beta_{\alpha(\gamma)}\mid\gamma<\eta\right\rangle $,
hence $\cof{\nu}=\cof{\eta}$. So without loss of generalization we
just assume that in fact $\nu=\eta$, and is regular and $<\mu$.
Note that this implies $\gamma\in x_{0}$ as $x_{0}$ is $<\mu$-closed. 

We claim that there must be some $\xi\in z$ such that $z\cap\xi$
is not $<\mu$ closed. If $i\in z\cap\theta$, then there is some
$\alpha<\eta$, an $\bar{\mathfrak{A}}$ term $\tau$ and $\vec{\alpha}\in x_{\alpha}\cap\gamma_{\alpha}$
such that $i=\tau(\vec{\alpha})\in x_{\alpha}$. This means that for
some $\rho\in x_{\alpha}\cap\mu$, $i<\bar{F}(\rho,0,\gamma)$. But
since $\rho<\gamma_{0}<\gamma_{\alpha}$ and $x_{\alpha}$ is in particular
a $\gamma_{0}$-end extension of $x_{0}$, $\rho\in x_{0}$. So in
fact $i<\sup\left(x_{0}\cap\theta\right)$. So we've shown that $\sup(z\cap\theta)=\sup(x_{0}\cap\theta)$.
$x_{0}$ is $<\mu$ closed so $\cof{\sup(x_{0}\cap\theta)}\geq\mu$,
so every sequence in $z\cap\theta$ of cofinality $<\mu$ must be
bounded by some $\xi\in z\cap\theta$, and in particular any sequence
witnessing that $z\cap\theta$ is not $<\mu$ closed, will be bounded
by some $\xi\in z$.

Let $\xi$ be the minimal ordinal in $z$ such that $z\cap\xi$ is
not $<\mu$ closed. Note that $z\prec\left\langle H(\theta),\in,\triangle\right\rangle $
so as in the previous lemma $\xi$ must be regular. In $z\cap\xi$
we have a sequence $\left\langle \beta_{\alpha}\mid\alpha<\eta\right\rangle $
such that it's supremum is not in $z\cap\xi$. By minimality this
sequence is cofinal in $z\cap\xi$. But if $\xi\in x_{\alpha}$, then
as before for every $i\in z\cap\xi$ there is $\rho\in x_{\alpha}$
such that $i<F(\rho,\xi,\gamma)$, so $\sup(z\cap\xi)=\sup(x_{\alpha}\cap\xi)$
and since $x_{\alpha}$ is $<\mu$-closed this is of cofinality $\geq\mu$,
by contradiction.
\end{proof}
We are now ready to prove our theorem.
\begin{proof}[Proof of theorem \ref{thm:embedding-properties}]
(1) This is a standard consequence of presaturation, see e.g. \cite[section 9]{Foreman-handbook}. 

(2) As we noted earlier every ordinal $\gamma<\delta$ is represented
by the function $Z\mapsto\otp\left(Z\cap\gamma\right)$. By concentrating
on stationary sets of subsets of size $<\lambda$, $\otp\left(Z\cap\gamma\right)<\lambda$,
so $\gamma<j(\lambda)$ for every $\gamma<\delta$, hence $\delta\leq j(\lambda)$,
so in particular $\crit j\leq\lambda$\,.

Let $\eta<\lambda$. Let $b\in\p$. There is $a\in\p$, $a\leq b$,
such that $\eta\in\cup a$. The set of $Z\in\power_{\lambda}\left(\cup a\right)$
such that $\eta\in Z$ is club, so we may assume $\forall Z\in a$,
$\eta\in Z$. By the definition of $\p$, $Z\cap\lambda\in\lambda$
so in fact $\eta\in Z$ implies $Z\cap\eta=\eta$. So the set of $d$
such that $\forall Z\in d$ $Z\cap\eta=\eta$ is dense in $\p$. Since
$\eta$ is represented by $\otp\left(Z\cap\eta\right)$, every such
$d$ forces $\eta=j\left(\eta\right)$, so by genericity, this is
the case. Thus $\lambda$ must be the critical point.

Assume towards contradiction that $j(\lambda)>\delta$. This means
that $\delta$ is represented by a function $f$ with domain $a\in V_{\delta}$
such that $\left[f\right]<\left[c_{\lambda}\right]$, so we can assume
$f:a\to\lambda$. We can also assume $\cup a=V_{\eta}$ for some strongly
inaccessible $\eta<\delta$. Since every $\gamma<\delta$ is represented
by the function $g_{\gamma}:Z\mapsto\otp(Z\cap\gamma)$, and $\left[f\right]>\left[g_{\gamma}\right]$,
we know that for every $\gamma<\delta$ there is some $a_{\gamma}\leq a$
such that for every $Z\in a_{\gamma}$, $f(Z\cap V_{\eta})>\otp(Z\cap\gamma)$.
We want to get a contradiction by finding some $a'\in\p$, $a'\leq a$
such that for some $\gamma$ and every $Z\in a'$, $f(Z\cap V_{\eta})\leq\otp(Z\cap\gamma)$.

Let $\rho\in\left(\eta,\delta\right)$ be a measurable cardinal (exists
from Woodinness) and fix a measure $U$ on $\rho$. Now we use lemmas
\ref{lem:mu-closed-extension} and \ref{lem:union-mu-closed}: Choose
some strongly inaccessible $\theta\in\left(\rho,\delta\right)$ and
let $\bar{\mathfrak{A}}$ be a Skolemized expansion of $\left\langle H(\theta),\in,\triangle,U,\left\{ \rho\right\} ,\left\{ \mu\right\} \right\rangle $.
Let $Z\prec\bar{\mathfrak{A}}$ with $\left|Z\right|<\lambda$, $Z\cap\lambda\in\lambda$
and consider $A_{Z}=\bigcap\left(U\cap Z\right)$. Since $\left|Z\right|<\lambda$,
$A_{Z}\in U$, although $Z$ doesn't know this fact, and in fact $A_{Z}\cap Z=\emp$
since $Z$ satisfies that $U$ is non-principal, i.e. for every $\zeta\in Z$,
$Z\vDash\exists A\in U\left(\zeta\notin A\right)$ so $\zeta\notin A_{Z}$.
Now fix some $\zeta\in A_{Z}$ and some function $G_{Z}$ from $Z\cap\lambda$
onto $Z$, and let $\bar{Z}$ be the substructure of $\overline{\left\langle \mathfrak{A},\left\{ \zeta\right\} ,G_{Z}\right\rangle }$
generated by $Z$. Let $\rho_{Z}=\sup(Z\cap\rho)$.
\begin{claim}
$\bar{Z}\cap\rho_{Z}=Z\cap\rho_{Z}$.
\end{claim}

\begin{proof}
Let $t$ be some Skolem term, $p\in Z^{<\omega}$ such that $\tau=t\left(p,\zeta\right)\in\bar{Z}\cap\rho_{Z}$.
 Consider the function $\xi\mapsto t(p,\xi)$ for $\xi<\rho$. We
have two cases
\begin{casenv}
\item The function is constant on some $A\in U$. So $H(\theta)\vDash\exists A\in U\forall\xi,\xi'\in A\left(t(p,\xi)=t(p,\xi')\right)$.
$Z\prec H(\theta)$ so there is such $A\in Z$. But then by the definition
of $A_{Z}$, $\zeta\in A$, so for some (any) $\xi\in A\cap Z$, $t(p,\xi)=t(p,\zeta)=\tau$
so $\tau\in Z$.
\item There is $\sigma<\rho$ such that $\left\{ \xi<\rho\mid t(p,\xi)\leq\sigma\right\} \in U$.
Then the function is regressive on the interval $\left(\sigma,\rho\right)\in U$
so by normality it is constant on some set in $U$, so we are back
to the previous case.
\item For every $\sigma<\rho$ we'd get $\left\{ \xi<\rho\mid t(p,\xi)>\sigma\right\} \in U$
. So this holds also in $Z$. In particular, if $\sigma=\min(Z\smin\tau)$,
we have $\left\{ \xi<\rho\mid t(p,\xi)>\sigma\right\} \in U\cap Z$
so $\zeta\in\left\{ \xi<\rho\mid t(p,\xi)>\sigma\right\} $, i.e.
$t(\rho,\zeta)=\tau>\sigma$, by contradiction to the choice of $\sigma$.
\end{casenv}
So the first case must hold, and we have $\tau\in Z$ as required.\qedhere

\end{proof}
Let $Z_{0}=Z$, and for every $\alpha\leq f(Z\cap V_{\eta})$, if
$Z_{\alpha}$ is defined and is, by induction, $<\mu$-closed, denote
$\rho_{\alpha}=Z_{\alpha}\cap\rho$, $Z_{\alpha}'=Z_{\alpha}\cap H(\rho_{\alpha})$,
$A_{Z_{\alpha}}=\bigcap\left(U\cap Z_{\alpha}\right)$, choose a function
$G_{Z_{\alpha}}:Z_{\alpha}\cap\lambda\twoheadrightarrow Z_{\alpha}$,
fix some $\zeta_{\alpha}\in A_{Z_{\alpha}}$ of cardinality $>\rho_{\alpha}$,
and let $Z_{\alpha+1}$ be the substructure of $\overline{\left\langle \mathfrak{A},\left\{ \zeta_{\alpha}\right\} ,G_{Z_{\alpha}}\right\rangle }$
generated by $Z_{\alpha}'$ . By the claim we have $Z_{\alpha+1}\cap\rho_{\alpha}=Z_{\alpha}\cap\rho_{\alpha}$,
in particular $Z_{\alpha+1}\cap H(\rho_{\alpha})=Z_{\alpha}\cap H(\rho_{\alpha})$,
so by lemma \ref{lem:mu-closed-extension} $Z_{\alpha+1}$ is $<\mu$
closed. Note that by including the predicate $G_{Z_{\alpha}}$ we
have $Z_{\alpha}\con Z_{\alpha+1}$. At limit stages we use lemma
\ref{lem:union-mu-closed} to get that $Z_{\alpha}=\bigcup_{\beta<\alpha}Z_{\beta}$
is $<\mu$ closed. Note that at every stage we have that $Z_{\alpha}\cap V_{\eta}=Z\cap V_{\eta}$.
So we get, after $f(Z\cap V_{\eta})$ stages, a set $Z_{*}=Z_{f(Z\cap V_{\eta})}$,
$Z_{*}\prec\bar{\mathfrak{A}}$, which is $<\mu$ closed, $Z_{*}\cap V_{\eta}=Z\cap V_{\eta}\in a$,
and $\otp(Z_{*}\cap\rho)\geq f(Z_{*}\cap V_{\eta})=f(Z\cap V_{\eta})$.

To conclude, what we have shown is that the set of all $Z\prec H(\theta)$
which are $<\mu$ closed, $\left|Z\right|<\lambda$, $Z\cap(\bigcup a)\in a$
and $\otp(Z\cap\rho)\geq f(Z\cap(\cup a))$ is stationary, which is
a contradiction. So $j(\lambda)=\delta$.

For (3), first note that since $\psu{<\delta}M\cap V\left[G\right]\con M$,
if $\mathrm{cf}^{V\left[G\right]}\left(\theta\right)<\delta$ or $\mathrm{cf}^{M}\left(\theta\right)<\delta$
then $M$ and $V\left[G\right]$ agree on this cofinality.

(a) If $\theta<\lambda$ is regular in $V$, then since $\lambda=\crit j$,
$\theta=j(\theta)$ remains regular in $M$, and is $<\delta$, so
it is also regular in $V\left[G\right]$. So regular cardinals $<\lambda$
from $V$ are preserved in $M$ and $V\left[G\right]$, hence all
cofinalities $<\lambda$ are preserved. 

(b) Assume $V\left[G\right],M\vDash\cof{\theta}<\mu$ ($\mu<\delta$
so if one satisfies this, so does the other). The size of the forcing
is $\delta>\mu$ and $\delta$ remains regular (since $j\left(\lambda\right)=\delta$
is regular in $M$, and if it were singular in $V\left[G\right]$,
then by closure of $M$ under $<\delta$-sequences it would have been
singular in $M$ as well) so we must have $V\vDash\cof{\theta}<\delta$.
By (a) if $\mu\leq\mathrm{cf}^{V}(\theta)<\lambda$ then it is preserved,
so we may assume (towards contradiction) that $\lambda\leq\mathrm{cf}^{V}\left(\theta\right)<\delta$,
and we can also assume $\mathrm{cf}^{V}(\theta)=\theta$. Let $a\Vdash M\vDash\cof{\theta}=\beta<\mu$.
This means that there is a function $f$ with domain $a$ such that
$a\Vdash$ ``$\left[f\right]$ is a sequence of length $\beta$ cofinal
in $\theta$''. $\beta<\crit j$ and $\theta$ is represented by
$\otp\left(\cdot\cap\theta\right)$ so without loss of generalization,
for every $Z\in a$ $f\left(Z\right)$ is a sequence of length $\beta$
cofinal in $\otp\left(Z\cap\theta\right)$. But $Z\cap\theta$ is
$<\mu$-closed and $\beta<\mu\leq\theta$ so it cannot have a cofinal
sequence of length $\beta$, a contradiction.

(c) $V\vDash\cof{\alpha}<\delta$ and $M\vDash\cof{\alpha}<\delta$
both imply $V\left[G\right]\vDash\cof{\alpha}<\delta$ since a cofinal
sequence from $V$ or $M$ will remain so in $V\left[G\right]$. $V\left[G\right]\vDash\cof{\alpha}<\delta$
implies $M\vDash\cof{\alpha}<\delta$ since $\psu{<\delta}M\cap V\left[G\right]\con M$.
Assume $V\left[G\right]\vDash\cof{\alpha}<\delta$. $\left|\p\right|=\delta$,
so it preserves cofinalities $>\delta$, so we must have $V\vDash\cof{\alpha}\leq\delta$.
If $V\vDash\cof{\alpha}=\delta$ while $V\left[G\right]\vDash\cof{\alpha}<\delta$,
it means that $\delta$ is not regular in $V\left[G\right]$ -- in
$V$ there is a $\delta$ sequence cofinal in $\alpha$, and in $V\left[G\right]$
there is also a shorter sequence cofinal in $\alpha$, which can be
pulled back to a cofinal sequence in $\delta$. But since $\psu{<\delta}M\cap V\left[G\right]\con M$
this means that $\delta$ is not regular in $M$ as well, contradicting
$j(\lambda)=\delta$ for $\lambda$ regular and $j$ elementary. So
we must have $V\vDash\cof{\alpha}<\delta$.
\end{proof}

\section{The main theorem}
\begin{thm}
Let $\lambda_{1}<\dots<\lambda_{n}<\delta_{1}<\dots<\delta_{n}$ such
that the $\lambda_{i}$s are regular uncountable cardinals and the
$\delta_{i}$s are Woodin cardinals. Then $C_{<\lambda_{1},\dots,<\lambda_{n}}^{*}\equiv C_{<\delta_{1},\dots,<\delta_{n}}^{*}$.
\end{thm}

\begin{proof}
Let $M_{0}=V$ and take some $\delta_{0}$ regular $\lambda_{n}<\delta_{0}<\delta_{1}$.
We define by induction generic elementary embeddings $j_{i}:M_{i-1}\to M_{i}$
for $i=1\til n$. Given $j_{1}\til j_{i-1}$, set $\bar{j}_{i-1}=j_{i-1}\circ\dots\circ j_{1}$
(if $i=1$ $\bar{j}_{0}=\id$) and $\bar{\lambda}_{l}=\bar{j}_{i-1}\left(\lambda_{l}\right)$
for $l\leq i$ and we inductively require that
\begin{enumerate}[label=({\arabic*})]
\item $\crit{j_{i}}=\bar{\lambda}_{i}$ and $j_{i}\left(\bar{\lambda}_{i}\right)=\delta_{i}$;
\item For $k>i$, $j_{i}\left(\delta_{k}\right)=\delta_{k}$;
\item For every $l=1\til n$ and every $\alpha$, $V\vDash\cof{\alpha}<\delta_{l}$
$\iff$ $M_{i}\vDash\cof{\alpha}<\delta_{l}$;
\item $M_{i}\vDash\delta_{i+1}\til\delta_{n}$ are Woodin.
\end{enumerate}
Let $\p_{i}=\p\left(\delta_{i},\bar{\lambda}_{i},<\delta_{i-1}^{+}\text{-closed}\right)^{M_{i-1}}$
(the stationary tower on $\delta_{i}$ consisting of stationary sets
with $<\delta_{i-1}^{+}$-closed elements of size $\bar{\lambda}_{i}$,
as computed in $M_{i-1}$), $G_{i}\con\p_{i}$ an $M_{i-1}$-generic
filter and $j_{i}:M_{i-1}\to M_{i}$ the associated embedding, i.e
we use the previous theorem with $V=M_{i-1}$, $\mu=\delta_{i-1}^{+}$
and $\lambda=\bar{\lambda}_{i}$. By elementarity of $\bar{j}_{i-1}$,
$\bar{\lambda}_{i}$ is a regular uncountable cardinal in $M_{i-1}$
so (1) follows from theorem \ref{thm:embedding-properties}.\ref{enu:crit},
(2) follows from lemma \ref{lem:inacc}.\ref{enu:fix-inacc} and
(4) is a standard fact regarding stationary towers.

For (3), first for $l<i$ by the induction hypothesis we have for
every $\alpha$, $V\vDash\cof{\alpha}<\delta_{l}$ $\iff$ $M_{i-1}\vDash\cof{\alpha}<\delta_{l}$.
By the assumptions,
\[
\delta_{l}=j_{l}(\bar{\lambda}_{l})=j_{l}\circ j_{l-1}\circ\dots\circ j_{1}(\lambda_{l})=\bar{j}_{l}\left(\lambda_{l}\right)<\bar{j}_{l}\left(\lambda_{i}\right)\leq\bar{j}_{i-1}\left(\lambda_{i}\right)=\bar{\lambda}_{i}.
\]
By theorem \ref{thm:embedding-properties}.\ref{enu:cof-preserve}(a),
all cofinalities $<\bar{\lambda}_{i}$ are preserved from $M_{i-1}$
to $M_{i}$, so in particular we have $V\vDash\cof{\alpha}<\delta_{l}$
$\Rightarrow$ $M_{i}\vDash\cof{\alpha}<\delta_{l}$. If $M_{i}\vDash\cof{\alpha}<\delta_{l}$,
then in particular $M_{i}\vDash\cof{\alpha}<\delta_{i-1}^{+}$, so
by theorem \ref{thm:embedding-properties}.\ref{enu:cof-preserve}(b)
also $M_{i-1}\vDash\cof{\alpha}<\delta_{i-1}^{+}\leq\bar{\lambda}_{i}$,
so again by \ref{thm:embedding-properties}.\ref{enu:cof-preserve}(a)
this cofinality is preserved from $M_{i-1}$ to $M_{i}$ hence $M_{i-1}\vDash\cof{\alpha}<\delta_{l}$
and by the induction hypothesis we get $V\vDash\cof{\alpha}<\delta_{l}$.
The case $l>i$ is an application of lemma \ref{lem:inacc}.\ref{lem:Cof<inacc}
to the induction hypothesis. For $l=i$, by the induction hypothesis
we have $V\vDash\cof{\alpha}<\delta_{i}$ $\iff$ $M_{i-1}\vDash\cof{\alpha}<\delta_{i}$,
and by theorem \ref{thm:embedding-properties}.\ref{enu:cof-preserve}(c)
we get $M_{i-1}\vDash\cof{\alpha}<\delta_{i}$ $\iff$ $M_{i}\vDash\cof{\alpha}<\delta_{i}$.

Now we look at $j=j_{n}\circ\dots\circ j_{1}:V\to M_{n}$. By the
construction we get that for every $i=1\til n$, 
\begin{align*}
j\left(\lambda_{i}\right) & =j_{n}\circ\dots\circ j_{i+1}\circ j_{i}\circ\dots\circ j_{1}\left(\lambda_{i}\right)\\
 & =j_{n}\circ\dots\circ j_{i+1}\circ j_{i}\left(\bar{\lambda}_{i}\right)\\
 & =j_{n}\circ\dots\circ j_{i+1}\left(\delta_{i}\right)=\delta_{i}
\end{align*}
so if we restrict $j$ to $\left(C_{<\lambda_{1},\dots,<\lambda_{n}}^{*}\right)^{V}$
we get an elementary embedding 
\[
\bar{j}:\left(C_{<\lambda_{1},\dots,<\lambda_{n}}^{*}\right)^{V}\to\left(C_{<j\left(\lambda_{1}\right),\dots,<j\left(\lambda_{n}\right)}^{*}\right)^{M_{n}}=\left(C_{<\delta_{1},\dots,<\delta_{n}}^{*}\right)^{M_{n}}.
\]
 From the last step in the induction we obtain that for every $i=1\til n$
and ordinal $\alpha$, $V\vDash\cof{\alpha}<\delta_{l}$ $\iff$ $M_{n}\vDash\cof{\alpha}<\delta_{l}$,
and since both $V$ and $M_{n}$ are contained in $V\left[G_{1}\til G_{n}\right]$
-- the least model containing $V$ and $G_{1}\til G_{n}$ -- in
this model we get that the construction of $\left(C_{<\delta_{1},\dots,<\delta_{n}}^{*}\right)^{M_{n}}$
and $\left(C_{<\delta_{1},\dots,<\delta_{n}}^{*}\right)^{V}$ will
yield the same results at each step, so indeed 
\[
\left(C_{<\lambda_{1},\dots,<\lambda_{n}}^{*}\right)^{V}\equiv\left(C_{<\delta_{1},\dots,<\delta_{n}}^{*}\right)^{M_{n}}=\left(C_{<\delta_{1},\dots,<\delta_{n}}^{*}\right)^{V}.\qedhere
\]
\end{proof}
\begin{cor}
If there is a proper class of Woodin cardinals, and $\lambda_{1}<\dots<\lambda_{n}$
are regular uncountable cardinals, then the theory of $C_{<\lambda_{1},\dots,<\lambda_{n}}^{*}$
is set-forcing absolute and does not depend on $\lambda_{1}\til\lambda_{n}$,
in the following sense: if $\p$ is a set-forcing, $G\con\p$ generic
and $\lambda'_{1}<\dots<\lambda'_{n}$ are in $V\left[G\right]$ regular
cardinals, then $\left(C_{<\lambda_{1},\dots,<\lambda_{n}}^{*}\right)^{V}\equiv\left(C_{<\lambda'_{1},\dots,<\lambda'_{n}}^{*}\right)^{V\left[G\right]}$.
\end{cor}

\begin{proof}
Let $\p$ be some forcing notion, and $\delta_{n}>\dots>\delta_{1}>\max\left(\left|\p\right|,\lambda_{n},\lambda_{n}'\right)$
Woodin cardinals. Note that after forcing with a generic $G\con\p$,
they remain Woodin. So we can apply the theorem both in $V$ and $V\left[G\right]$
to obtain 
\begin{align*}
\left(C_{<\lambda_{1},\dots,<\lambda_{n}}^{*}\right)^{V} & \equiv\left(C_{<\delta_{1},\dots,<\delta_{n}}^{*}\right)^{V}\\
\left(C_{<\lambda'_{1},\dots,,\lambda'_{n}}^{*}\right)^{V\left[G\right]} & \equiv\left(C_{<\delta_{1},\dots,<\delta_{n}}^{*}\right)^{V\left[G\right]}
\end{align*}
but being of cofinality $<\delta_{i}$ is not affected by a forcing
of size $<\delta_{1}$, so $\left(C_{<\delta_{1},\dots,<\delta_{n}}^{*}\right)^{V}=\left(C_{<\delta_{1},\dots,<\delta_{n}}^{*}\right)^{V\left[G\right]}$,
and we get $\left(C_{<\lambda_{1},\dots,<\lambda_{n}}^{*}\right)^{V}\equiv\left(C_{<\lambda'_{1},\dots,<\lambda'_{n}}^{*}\right)^{V\left[G\right]}$.
\end{proof}
\begin{rem}
\begin{enumerate}
\item Note that $\cof{\alpha}=\omega\iff\cof{\alpha}<\omega_{1}$ so for
$\lambda_{1}=\omega_{1}$ we get the original result for $C_{\omega}^{*}$.
\item For any regular $\lambda$, $\cof{\alpha}=\lambda\iff\left(\cof{\alpha}<\lambda^{+}\land\neg\cof{\alpha}<\lambda\right)$.
For any regular $\lambda_{1}<\lambda_{2}$, denote by $C_{\left[\lambda_{1},\lambda_{2}\right)}^{*}$
the model constructed with the logic obtained from first order logic
by adding the quantifier $Q_{\left[\lambda_{1},\lambda_{2}\right)}^{\mathrm{cf}}(\alpha)\iff\cof{\alpha}\in\left[\lambda_{1},\lambda_{2}\right)$.
We can write $C_{\lambda}^{*}$ as $C_{\left[\lambda,\lambda^{+}\right)}^{*}$.
Now, if $\lambda$ is regular uncountable, the proof of the main theorem
using $\lambda,\lambda^{+}$ and Woodins $\delta_{1}<\delta_{2}$,
also gives us a generic elementary embedding
\[
\bar{j}:\left(C_{\left[\lambda,\lambda^{+}\right)}^{*}\right)^{V}\to\left(C_{\left[\delta_{1},\delta_{2}\right)}^{*}\right)^{V}
\]
so the proof of the corollary can be applied to obtain the absoluteness
of $Th\left(C_{\lambda}^{*}\right)$.
\item In general, we get absoluteness of $Th\left(C_{\lambda_{1}\til\lambda_{n}}^{*}\right)$
for any regular cardinals $\lambda_{i}$. 
\end{enumerate}
\end{rem}

\section{Open questions}

A natural question at this stage would be whether our results are
true also for infinitely many cofinality quantifiers. However, the
naïve approach of taking a direct limit of the above construction
and trying to prove similar results does not work, as cofinalities
are not presereved to the limit stage. Let $\vec{\lambda}=\left\langle \lambda_{n}\mid n<\omega\right\rangle $
be an increasing sequence of regular cardinals and $\left\langle \delta_{n}\mid n<\omega\right\rangle $
an increasing sequence of Woodin cardinals such that $\delta_{0}>\lambda_{*}:=\sup\left\langle \lambda_{n}\mid n<\omega\right\rangle $.
For every $i<\omega$ define by induction generic elementary embeddings
$j_{i,i+1}:M_{i}\to M_{i+1}\con V\left[G_{1}\til G_{i+1}\right]$
($M_{0}=V$) as in the proof of the main theorem (setting $j_{0,0}=\id$,
$j_{0,n+1}=j_{n,n+1}\circ j_{0,n}$\,, $G_{i+1}$ is $M_{i}$-generic
for the poset $\p\left(\delta_{i},j_{0,i}(\lambda_{i}),<\delta_{i-1}^{+}\text{-closed}\right)^{M_{i}}$).
We get that $j_{0,n+1}(\lambda_{n})=\delta_{n}$, for every $m\ne n+1$
$j_{0,m}(\delta_{n})=\delta_{n}$, and for every $k<\omega$ and $\eta\in\ord$,
$V\vDash\mathrm{cf}(\eta)<\delta_{k}$$\iff$$M_{n}\vDash\mathrm{cf}(\eta)<\delta_{k}$.
We can now take the direct limit and get embeddings $j_{n,\omega}:M_{n}\to M_{\omega}$
which are defined in the finite support iteration of the forcings
which we denote by $V\left[\mathcal{G}\right]$. So we get that for
every $n$ $j_{0,\omega}(\lambda_{n})=\delta_{n}$ and for every $m>n+1$
$j_{m,\omega}(\delta_{n})=\delta_{n}$. 

Denote for every $n\leq\omega$ $\theta_{n}:=j_{0,n}(\left(\lambda_{*}^{+}\right)^{V})$,
and consider $\theta_{\omega}$. First note that 
\[
j_{0,\omega}(\lambda_{*})=\sup j_{0,\omega}\left(\left\langle \lambda_{i}\mid i<\omega\right\rangle \right)=\sup\left\langle j_{0,\omega}\left(\lambda_{i}\right)\mid i<\omega\right\rangle =\sup\left\langle \delta_{i}\mid i<\omega\right\rangle 
\]
 so if we denote this by $\delta_{*}$ we get by elementarity that
$\theta_{\omega}=j_{0,\omega}(\lambda_{*}^{+})=\left(\delta_{*}^{+}\right)^{M_{\omega}}$.
So in particular $M_{\omega}\vDash\cof{\theta_{\omega}}=\theta_{\omega}>\delta_{*}$,
or equivalently (since $\delta_{*}=\sup\left\langle \delta_{n}\mid n<\omega\right\rangle $)
for every $n$ $M_{\omega}\vDash\cof{\theta_{\omega}}>\delta_{n}$
. If this were to hold also in $V$, then it would hold also in $V\left[\mathcal{G}\right]$,
since it is a forcing extension with a forcing of size at most $\delta_{*}$.
But in fact $V\left[\mathcal{G}\right]\vDash\cof{\theta_{\omega}}=\omega$:
$\theta_{\omega}=j_{n,\omega}(\theta_{n})$ for every $n$, and by
definition of the direct limit, every $\eta<\theta_{\omega}$ is of
the form $j_{n,\omega}(\bar{\eta})$ for some $n<\omega$ and $\bar{\eta}<\theta_{n}$,
so the sequence $\left\langle \sup j_{n,\omega}\left[\theta_{n}\right]\mid n<\omega\right\rangle $
is cofinal in $\theta_{\omega}$.  We claim that this sequence is
strictly increasing. For every $n$, 
\[
\theta_{n}=j_{0,n}(\lambda_{*}^{+})<j_{0,n}(\delta_{n})=\delta_{n}.
\]
$M_{n+1}$ is closed under $<\delta_{n}$ sequences in $M_{n}\left[G_{n}\right]$,
so $j_{n,n+1}\left[\theta_{n}\right]\in M_{n+1}$ and is of ordertype
$\theta_{n}<\delta_{n}$, while $\theta_{n+1}=j_{n,n+1}(\theta_{n})>j_{n,n+1}(\lambda_{n})=\delta_{n}$,
so $j_{n,n+1}\left[\theta_{n}\right]$ is bounded below $\theta_{n+1}$
so also $\sup j_{n,\omega}\left[\theta_{n}\right]$ is below $\sup j_{n+1,\omega}\left[\theta_{n+1}\right]$.
Hence in $V\left[\mathcal{G}\right]$ $\cof{\theta_{\omega}}=\omega$.

So to conclude, the question of absoluteness of the theory of the
model constructed with infinitely many cofinality quantifers remains
open.

\bibliographystyle{amsplain}
\bibliography{Bibliography}

\end{document}